\documentclass{amsart}
 \usepackage{amssymb, latexsym}

\DeclareMathOperator{\ar}{\mathrm{Area}}
\DeclareMathOperator{\len}{\mathrm{Length}}

\begin{document}
 \bibliographystyle{plain}

 \newtheorem*{definition}{Definition}
 \newtheorem{theorem}{Theorem}
 \newtheorem{lemma}{Lemma}
 \newtheorem{corollary}{Corollary}
 \newcommand{\mc}{\mathcal}
 \newcommand{\A}{\mc{A}}
 \newcommand{\B}{\mc{B}}
 \newcommand{\F}{\mc{F}}
 \newcommand{\T}{\mc{T}}
 \newcommand{\lf}{\left\lfloor}
 \newcommand{\rf}{\right\rfloor}
 \newcommand{\rar}{\rightarrow}
 \newcommand{\mbb}{\mathbb}
 \newcommand{\R}{\mbb{R}}
 \newcommand{\N}{\mbb{N}}
 \newcommand{\Q}{\mbb{Q}}
 \newcommand{\Z}{\mbb{Z}}
 \newcommand{\Zv}{\Z^2_{\text{vis}}}
\title[Farey fractions with denominators in arithmetic progressions]{A note on Farey fractions with\\ denominators in arithmetic progressions}
\author{Dmitry A. Badziahin and Alan K. Haynes}
\subjclass[2000]{11B57} \keywords{Farey fractions}
\thanks{DB:~Research supported by EPSRC grant
EP/E061613/1.}
\thanks{AH:~Research supported by EPSRC grant EP/F027028/1.}

 \allowdisplaybreaks


\begin{abstract}
As the conclusion of a line of investigation undertaken in \cite{haynes2003}, we compute asymptotic frequencies for the values taken by numerators of differences of consecutive Farey fractions with denominators restricted to lie in arithmetic progression.
\end{abstract}


\maketitle

\section{Introduction}

For $Q\in\mathbb{N}$ the Farey fractions of order $Q$ are defined as \[\F_Q=\left\{\frac{a}{q}\in\Q : 1\leq q\leq Q, 0< a\le q,  (a,q)=1\right\},\] where $(a,q)$ denotes the greatest common divisor of $a$ and $q$. They are assumed to be ordered in the natural way on the unit interval. For each $d\in\N$ we define a subset $\F_{Q,d}\subseteq\F_Q$ by
\[\F_{Q,d}=\left\{\frac{a}{q}\in\F_Q : (q,d)=1\right\}.\]
Farey fractions play an important role in mathematics, computer science, and physics. For example there are many statements about the distribution of the Farey sequence which are known to be equivalent to the Riemann Hypothesis (\cite{franel1924},\cite{landau1924},\cite{kanemitsu1996}). Also knowledge about Farey fractions has led to significant progress in the study of billiards (\cite{boca2003}). By comparison, the restricted sets $\F_{Q,d}$ are known to have applications to the Generalized Riemann Hypothesis (\cite{huxley1971}) and to billiards in which the source is a point with non-zero rational coordinates (\cite[Introduction]{alkan2006b}).

A fundamental problem is to understand the distribution of gaps between consecutive elements of the sets $\F_{Q,d}$. Toward this end for each $k\in\N$ we define
\[N_{Q,d}(k)=\#\left\{\frac{a}{q}<\frac{a'}{q'}\text{ consecutive in } \F_{Q,d} : qa'-aq'=k\right\}.\] In this paper we prove the following theorem.
\begin{theorem}\label{mainthm}
For any $d>1$ and $k\in\N$ there exists a positive constant $c(d,k)$ for which
\begin{equation*}
N_{Q,d}(k)=c(d,k)Q^2+O_{d,k}(Q\log Q)
\end{equation*}
as $Q\rar\infty$.
\end{theorem}
The case when $d$ is a prime number has been previously studied in \cite{haynes2003} and \cite{haynes2004}. When $d$ is composite the situation is somewhat more complicated. This is because when $d$ is prime at least one out of every two consecutive fractions in $\F_Q$ must also be an element of $\F_{Q,d}$ (a crucial fact in the proofs of \cite[Theorem 1]{haynes2003} and \cite[Theorem 1]{haynes2004}). However when $d$ is composite and when $Q$ is large enough there are longer sequences of consecutive fractions in $\F_Q$ which do not appear in $\F_{Q,d}$. In fact it is easy to see that by letting $Q$ and $d$ vary we can find sequences of this kind which are arbitrarily long. For a naive example of how this can happen, we could let $d=Q!$ and then take $Q$ to be as large as we like. Interestingly this appears not to have been an issue in the recent related papers \cite{alkan2006a} and \cite{alkan2006b}.

We are able to overcome the phenomenon of large gaps by using properties of the $\ell-$index of the Farey sequence, introduced and studied in \cite{haynes2008}. Let us write \[\F_Q=\{\gamma_1,\gamma_2,\ldots ,\gamma_{N(Q)}\}\] with
$1/Q=\gamma_1<\gamma_2<\cdots <\gamma_{N(Q)}=1.$ It is also natural to extend this sequence by requiring that $\gamma_{i+N(Q)}=\gamma_i+1$ for
all $i\in\Z$. Then for each $i$ we write $\gamma_i=a_i/q_i$ with
$(a_i,q_i)=1$ and $q_i>0$. Given a positive integer $\ell$ and a fraction $\gamma_i$ in $\F_Q$
the $\ell-$index of $\gamma_i$ is defined to be the quantity
\begin{equation*}
\nu_\ell(\gamma_i)=a_{i+\ell-1}q_{i-1}-a_{i-1}q_{i+\ell-1}.
\end{equation*}
To avoid confusion note that the definition of $\nu_\ell$ is dependent on $Q$. The main theorem in \cite{haynes2008} is an asymptotic formula for the average value of $\nu_\ell$, as $Q\rar\infty$. In the course of proving that theorem it was also necessary to establish algebraic identities which allow one to express the $\ell$-index as a multivariable polynomial evaluated at the values taken by the $2-$index. To be precise
defined $K_0\in\Z$ and $K_1\in\Z[x_1]$ by
\begin{align*}
K_0(\cdot )=1\quad\text{and}\quad K_1(x_1)=x_1,
\end{align*}
and then for each $\ell\ge 2$ define $K_\ell\in\Z [x_1,x_2,\ldots
,x_\ell]$
by
\begin{align*}
 K_\ell(x_1,x_2,\ldots,x_\ell)=&x_\ell K_{\ell-1}(x_1,x_2,\ldots
,x_{\ell-1})+K_{\ell-2}(x_1,x_2,\ldots ,x_{\ell-2}).
\end{align*}
These polynomials are known as the convergent polynomials, and they appear in the study of continued fractions (see \cite[Section 6.7]{graham1994}). In \cite[Theorem 1]{haynes2008} it was proved that for any $Q,\ell\in\N$ and for any $\gamma_i\in\F_Q$ we have that
\begin{equation}\label{contiden1}
\nu_\ell(\gamma_i)=\left(\frac{2\ell-1}{2}\right)K_{\ell-1}\left(-\nu_2(\gamma_i),\nu_2(\gamma_{i+1}),\ldots
,(-1)^{\ell-1}\nu_2(\gamma_{i+\ell-2})\right),
\end{equation}
where $(\frac{\cdot}{2})$ is the Kronecker symbol, defined by
\begin{equation*}
\left(\frac{\ell}{2}\right)=\begin{cases}0&\text{if }2|\ell,\\1&\text{if }
\ell\equiv\pm 1\mod 8,\text{ and}\\-1&\text{if }\ell\equiv\pm 3\mod
8.\end{cases}
\end{equation*}
The $\ell-$index tells us what happens to the numerator of the difference of the endpoints of a consecutive $(\ell+1)-$tuple in $\F_Q$, while the $2-$index has a geometrical description in terms of an area-preserving surjection of a certain region in $\R^2$. In what follows we will show how these facts combine with (\ref{contiden1}) to give us much of what we need to prove Theorem \ref{mainthm}.


\section{Notation and definitions}
In order to take advantage of identity (\ref{contiden1}) we will use a geometrical description of the Farey sequence and a convenient area preserving map. These tools have been previously used by several authors (e.g. \cite{boca2001}, \cite{boca2002}, \cite{haynes2008})

The Farey triangle $\T\subseteq[0,1]^2$ is defined by
\begin{equation*}
\T=\{(x,y)\in [0,1]^2:x+y>1\}
\end{equation*}
and the Farey map $T:[0,1]^2\rar [0,1]^2$ is defined by
\begin{equation*}
T(x,y)=\left(y,\left[\frac{1+x}{y}\right]y-x\right),
\end{equation*}
where $[x]$ denotes the greatest integer less than or equal to $x$. The map $T$ is a one-to-one area preserving transformation of $\T$ onto itself (\cite{boca2001}) and has the important property that
\begin{equation}\label{Tkform2}
T\left(\frac{q_{j-1}}{Q},\frac{q_j}{Q}\right)=\left(\frac{q_j}{Q},\frac{q_{j+1}}{Q}\right).
\end{equation}
Also for each positive integer $k$ we let
\begin{equation*}
\T_k=\left\{(x,y)\in\T:\left[\frac{1+x}{y}\right]=k\right\}.
\end{equation*}
Note that the set $\T$ is the disjoint union of the sets $\T_k$ and that
\begin{equation*}
T(x,y)=(y,ky-x)\quad\text{ for all }\quad(x,y)\in\T_k.
\end{equation*}
Also it follows from \cite[Equation 1.4]{hall2003} that given $Q$ and $i$,
\begin{equation}\label{Tkform1}
\nu_2(\gamma_i)=k~\text{ if and only if }~\left(\frac{q_{i-1}}{Q},\frac{q_i}{Q}\right)\in\T_k.
\end{equation}
In the next section we will need to know how frequently large values of the Farey index can be obtained. The following lemma answers this question by showing that if the $2-$index of a particular fraction is large, then the $2-$indices of the fractions close by must be small.
\begin{lemma}\label{largeindlem1}
If $Q\in\N$, $k\ge 1$ and $\nu_2(\gamma_i)\ge 4k+1$ then we have that
\[\nu_2(\gamma_{i\pm 1})=1~\text{ and that }~\nu_2(\gamma_{i\pm j})=2\]
for $2\le j\le k.$
\end{lemma}
\begin{proof}
Since $4k+5=4(k+1)+1$, it is enough to verify that the theorem is true when $\nu_2(\gamma_i)=4k+j$ with $j=1,2,3,$ or $4$.

Suppose that the fractions $a'/q'<\gamma_i <a''/q''$ are consecutive in $\F_{q_i}$ and for each positive integer $m$ write
\begin{eqnarray*}
b_m=a'+ma_i,& &r_m=q'+mq_i\\
c_m=a''+ma_i,&\text{ and } &s_m=q''+mq_i.
\end{eqnarray*}
Then we have that $(b_m,r_m)=(c_m,s_m)=1$ for all $m$.

First we consider the case when $\nu_2(\gamma_i)=4k+1$. In this case we claim that the fractions
\begin{equation*}
\frac{b_k}{r_k}<\cdots <\frac{b_{2k-1}}{r_{2k-1}} <\frac{b_{2k}}{r_{2k}}<\frac{a_i}{q_i}<\frac{c_{2k}}{s_{2k}}<\frac{c_{2k-1}}{s_{2k-1}}<\cdots <\frac{c_k}{s_k}
\end{equation*}
are consecutive in $\F_Q$. To see this first let $M$ and $N$ be the largest positive integers for which $r_M,s_N\le Q$. Then we must certainly have that $r_M, s_N> Q-q_i$ and thus that $|r_M-s_N|<q_i$ and $|M-N|\le 1$. By a well known formula for the $2-$index (\cite[Equation 1.2]{hall2003}) we also have that
\[\nu_2(\gamma_i)=\frac{r_M+s_N}{q_i}=M+N+1\] and since $\nu_2(\gamma_i)=4k+1$ this implies that $M\equiv N\mod 2$ and we deduce that $M=N=2k$.

Now notice that the fractions $b_{2k}/r_{2k}<a_i/q_i$ are consecutive in $\F_{r_{2k}}$ and by the mediant property of Farey fractions the fraction with smallest denominator which lies between them is $b_{2k+1}/r_{2k+1}.$ Since we have already checked that $r_{2k+1}>Q$ this shows that $b_{2k}/r_{2k}$ and $a_i/q_i$ are consecutive in $\F_Q$. Similarly for $0\le i\le k-1$ the fractions $b_{k+i}/r_{k+i}<b_{k+i+1}/r_{k+i+1}$ are consecutive in $\F_{r_{k+i+1}}$ and the fraction with smallest denominator which lies between them is \[(b_{k+i}+b_{k+i+1})/(r_{k+i}+r_{k+i+1}).\]
Since \[r_{k+i}+r_{k+i+1}=2q'+(2(k+i)+1)q_i>r_{2k+1}>Q\] this shows that $b_{k+i}/r_{k+i}$ and $b_{k+i+1}/r_{k+i+1}$ are consecutive in $\F_Q$. The analysis of the fractions to the right of $a_i/q_i$ is identical and this verifies our above claim. The rest of the proof in this case follows from a straightforward calculation.

In the case when $\nu_2(\gamma_i)=4k+2$ we apply the same arguments to deduce that either the fractions
\begin{equation*}
\frac{b_{k+1}}{r_{k+1}}<\cdots <\frac{b_{2k}}{r_{2k}} <\frac{b_{2k+1}}{r_{2k+1}}<\frac{a_i}{q_i}<\frac{c_{2k}}{s_{2k}}<\frac{c_{2k-1}}{s_{2k-1}}<\cdots <\frac{c_k}{s_k}
\end{equation*}
or the fractions
\begin{equation*}
\frac{b_k}{r_k}<\cdots <\frac{b_{2k-1}}{r_{2k-1}} <\frac{b_{2k}}{r_{2k}}<\frac{a_i}{q_i}<\frac{c_{2k+1}}{s_{2k+1}}<\frac{c_{2k}}{s_{2k}}<\cdots <\frac{c_{k+1}}{s_{k+1}}
\end{equation*}
are consecutive in $\F_Q$. Either way the lemma is verified, and the two remaining cases are dealt with in the same manner.
\end{proof}
We will also need estimates for the number of lattice points in a convex region satisfying certain congruence constraints.
\begin{lemma}\label{latptlem1}
Suppose that $\Omega$ is a convex region contained in the square $[0,Q]\times [0,Q]$, that $\A$ and $\B$ are subsets of $\{1,\ldots ,d\}$, and that each element $a$ of $\A$ satisfies $(a,d)=1$. Then for the quantity
\[Z_{\A,\B,d}(\Omega )=\#\{(m,n)\in\Omega\cap\Zv : (m,n)\equiv (a,b)\mod d\text{ for some }(a,b)\in\A\times\B\}\]
we have the estimate
\begin{equation*}
Z_{\A,\B,d}(\Omega )=\frac{6\ar (\Omega)|\A||\B|}{\pi^2d^2}\prod_{p|d}\left(1-\frac{1}{p^2}\right)^{-1}+O\left(\frac{\ar (\Omega)}{Q}+\len (\partial\Omega)\log Q\right),
\end{equation*}
where the product here is taken over primes dividing $d$.
\end{lemma}
\begin{proof}
First of all by M\"{o}bius inversion we have that
\begin{align*}
Z_{\A,\B,d}(\Omega)&=\sum_{\substack{(m,n)\in\Omega\\(m,n)\mod d\in\A\times\B}}\sum_{e|m,n}\mu (e)\\
&=\sum_{e\le Q}\mu (e)\sum_{\substack{(m,n)\in\Omega,~e|m,n\\(m,n)\mod d\in\A\times\B}}1.
\end{align*}
Since the residue classes in $\A$ are coprime to $d$ the inner sum on the right hand side here will be zero if $(e,d)>1$. Thus the right hand side is equal to
\begin{align*}
\sum_{\substack{e\le Q\\(e,d)=1}}\mu (e)\sum_{\substack{(m,n)\in\Omega\\(m,n)\mod de\in\A_e\times\B_e}}1,
\end{align*}
where $\A_e,\B_e\subseteq\{1,\ldots ,de\}$ are uniquely determined (i.e. by the Chinese Remainder Theorem) and satisfy $|\A_e|=|\A|$ and $|\B_e|=|\B|$. Now by using a standard estimate to count the lattice points in the innermost sum we find that the above expression equals
\begin{align*}
&\sum_{\substack{e\le Q\\(e,d)=1}}\mu (e)\left(\frac{\ar (\Omega)|\A||\B|}{(de)^2}+O\left(\frac{\len (\partial (\Omega))\max\{|\A|,|\B|\}}{de}\right)\right)\\
&\quad =\frac{\ar (\Omega)|\A||\B|}{d^2}\sum_{\substack{e\le Q\\(e,d)=1}}\frac{\mu (e)}{e^2}+O\left(\len (\partial (\Omega))\log Q\right)\\
&\quad =\frac{6\ar (\Omega)|\A||\B|}{\pi^2d^2}\prod_{p|d}\left(1-\frac{1}{p^2}\right)^{-1}+O\left(\frac{\ar (\Omega)}{Q}+\len (\partial\Omega)\log Q\right).
\end{align*}
The final equality here comes from the fact that
\begin{align*}
\sum_{\substack{e\le Q\\(e,d)=1}}\frac{\mu (e)}{e^2}=\sum_{e=1}^\infty\frac{\mu (e)}{e^2}+O\left(\frac{1}{Q}\right)=\zeta(2)^{-1}\prod_{p|d}\left(1-\frac{1}{p^2}\right)^{-1}+O\left(\frac{1}{Q}\right).
\end{align*}
\end{proof}


\section{Proof of Theorem \ref{mainthm}}
For each positive integer $\ell$ let us define
\begin{align*}
N_{Q,d}(k,\ell)=\#\{&1\le i\le N(Q) : (q_{i-1},d)=(q_{i+\ell-1},d)=1,\\
&\qquad (q_{i+j-1},d)>1 \text{ for } 1\le j<\ell, ~\nu_\ell(\gamma_i)=k\}.
\end{align*}
Then it is clear that
\begin{equation}\label{Nksum1}
N_{Q,d}(k)=\sum_{\ell=1}^{N(Q)-1}N_{Q,d}(k,\ell).
\end{equation}
Our proof will consist of the following three steps:
\begin{itemize}
\item[i.] We will show that there is an integer $L=L(d)$ such that for any $Q$ and $i$ we have $(q_{i+j-1},d)=1$ for some $1\le j\le L$.
\item[ii.] We will show that for each $\ell$ and $k$ there is an integer $K=K(\ell ,k)$ such that if $\nu_\ell(\gamma_i)=k$ for some $Q$ and $i$ then the $(\ell-1)-$tuple $(\nu_2(\gamma_i),\ldots , \nu_2(\gamma_{i+\ell-2}))$ must take one of at most $K^{\ell-1}$ possible values.
\item[iii.] We will use the information gathered in the first two steps to estimate (\ref{Nksum1}) by counting visible lattice points in subregions of the Farey triangle which satisfy our congruence constraints.
\end{itemize}
\subsection*{Step i}
First suppose that $Q$ and $i_1<i_2$ are chosen so that
\begin{equation}\label{consecfrac1}
\max\{q_{i_1},q_{i_2}\}\le q_j ~\text{ and  }~ (q_j,d)>1 ~\text{ for all }~ i_1\le j\le i_2.
\end{equation}
It follows easily from this that $\gamma_{i_1}$ and $\gamma_{i_2}$ are consecutive in the Farey fractions of order $\max\{q_{i_1},q_{i_2}\}$. Thus by the mediant property of the Farey sequence we have that
\begin{equation}\label{consecfrac2}
\{q_j:i_1<j<i_2\}=\{mq_{i_1}+nq_{i_2}:m,n\in\N, (m,n)=1, mq_{i_1}+nq_{i_2}\le Q\}.
\end{equation}
Now let $d_1$ be the largest divisor of $d$ which is coprime to $q_{i_1}$ and choose $m_1\in\N$ so that $m_1\le d_1$ and \[m_1\equiv (1-q_{i_2})q_{i_1}^{-1}\mod d_1.\] Then for any prime $p$ which divides $d_1$ we have that $p\nmid m_1q_{i_1}+q_{i_2}$. On the other hand since $\gamma_{i_1}$ and $\gamma_{i_2}$ are consecutive Farey fractions of some order we have that $(q_{i_1},q_{i_2})=1$ and so if $p|(d/d_1)$ then it follows that $p\nmid q_{i_2}$. This shows that $(m_1q_{i_1}+q_{i_2},d)=1$ and it follows from (\ref{consecfrac1}) and (\ref{consecfrac2}) that $Q<m_1q_{i_1}+q_{i_2}\le dq_{i_1}+q_{i_2}$. By repeating the same argument we also have that $Q<q_{i_1}+dq_{i_2}$ and so
\begin{equation}\label{consecfrac3}
\{q_j:i_1<j<i_2\}\subseteq \{mq_{i_1}+nq_{i_2}:1\le m,n< d\}.
\end{equation}
From this it is apparent that if (\ref{consecfrac1}) is satisfied then we must have $i_2-i_1< d^2$.

Now suppose more generally that $Q, i,$ and $L_0$ are chosen so that
\[(q_{i+j-1},d)>1~\text{ for all }~1\le j\le L_0.\]
Choose an integer $i_0\in\{i,\ldots ,i+L_0-1\}$ such that
\[q_{i_0}=\min\{q_j:i\le j\le i+L_0-1\}.\]
This integer uniquely determines a pair of non-negative integers $n$ and $n'$ and a finite sequence $i_{-n}<\cdots<i_0<\cdots <i_{n'}$ with the properties that $i_{-n}=i, i_{n'}=i+L_0-1,$
\begin{align*}
q_{i_m}&=\min\{q_j:i_{m-1}< j\le i+L_0-1\}~\text{ for }1\le m\le n',\\
q_{i_m}&=\min\{q_j:i\le j\le i_{m+1}\}~\text{ for }-n\le m\le -1,
\end{align*}
and
\begin{equation*}
\max\{q_{i_m},q_{i_{m+1}}\}< q_j~\text{ for all }~-n\le m<n'~\text{ and }~i_m< j< i_{m+1}.
\end{equation*}
Our argument from above shows that $i_{m+1}-i_m<d^2$ for each $-n\le m<n'$, so all that remains is to bound $n$ and $n'$ in terms of $d$.

If $q_{i_m}$ were equal to $q_{i_{m+1}}$ for some $m$ then there would have to be a fraction with smaller denominator between $a_{i_m}/q_{i_m}$ and $a_{i_{m+1}}/q_{i_{m+1}}$. Since our setup does not allow this we must have that
\[q_{i_0}<q_{i_{-1}}<\cdots <q_{i_{-n}}~\text{ and }~q_{i_0}< q_{i_1}<\cdots< q_{i_{n'}}.\]
Now for each integer $0\le m< n'$ let $b_m/r_m$ be the fraction which immediately follows $a_{i_m}/q_{i_m}$ in the Farey fractions of order $q_{i_{m+1}}-1$. It is evident that $r_0\le r_1\le\cdots \le r_{n'-1}$ and that $q_{i_{m+1}}=q_{i_m}+r_m$ for each $0\le m<n'$. Furthermore it is not difficult to see that if $r_m>r_{m-1}$ for some $m$ then $r_m>q_{i_m}$. Suppose that $r_0=r_m$ for all $m< m'$. Then we have that \[q_{i_{m'}}=q_{i_0}+m'r_0\]
and making use of the argument used to prove (\ref{consecfrac3}) we find that $m'<d$. Thus if $n'> d$ then for some $1\le m\le d$ we must have that $r_m>r_{m-1}$. Then since $r_m>q_{i_m}> q_{i_0}$ and \[q_{i_{n'}}\ge q_{i_m}+(n'-m)r_m\]
we have that
\[q_{i_{n'}}> q_{i_0}+(n'-m)q_{i_1}.\]
Again by the argument used to prove (\ref{consecfrac3}) we have that $n'-m<d$ which allows us to conclude that $n'<2d$. Similarly we find that $n<2d$ and putting these results together shows that $L_0<L(d)=4d^3$.

\subsection*{Step ii}
In this step we are supposing that $\ell$ and $k$ are fixed and we want to show that if $\nu_\ell(\gamma_i)=k$ for some $Q$ and $i$ then there are only finitely many possible values for the $(\ell-1)-$tuple of integers $(\nu_2(\gamma_i),\ldots , \nu_2(\gamma_{i+\ell-2}))$.

First suppose that $K_0>\ell$ and that $\nu_2(\gamma_{i_0})=4K_0+1$ for some $i\le i_0\le i+\ell-2$. Then by Lemma \ref{largeindlem1} we know that $\nu_2(\gamma_{i_0\pm 1})=1$ and that $\nu_2(\gamma_{i_0\pm j})=2$ for $2\le j\le\ell$. Furthermore following the proof of the lemma if we suppose that \[\frac{a'}{q'}<\frac{a_{i_0}}{q_{i_0}}<\frac{a''}{q''}\] are consecutive in $\F_{q_{i_0}}$ then it is easy to see that
\[\gamma_{i-1}=\frac{a'+(2K_0-(i_0-i))a_{i_0}}{q'+(2K_0-(i_0-i))q_{i_0}}\]
and that \[\gamma_{i+\ell-1}=\frac{a''+(2K_0-(i+\ell-2-i_0))a_{i_0}}{q''+(2K_0-(i+\ell-2-i_0))q_{i_0}}.\]
Using the definition of $\nu_\ell$ we have that
\begin{equation}\label{indeqn1}
\nu_{\ell}(\gamma_i)=2K_0(q_{i_0}a''+a_{i_0}q'-q_{i_0}a'-a_{i_0}q'')+C=4K_0+C,
\end{equation}
where the constant $C$ here is an integer which depends only on $a_{i_0},q_{i_0},$ and $i$.

We know from identity (\ref{contiden1}) that $\nu_{\ell}(\gamma_i)$ is a linear function of each of the variables $\nu_2(\gamma_{i+j}),~0\le j\le \ell-2$, and this together with Lemma \ref{largeindlem1} allows us to interpolate the values of $\nu_2(\gamma_{i_0})$ between integers of the form $4K_0+1$. Since the right hand side of (\ref{indeqn1}) tends to infinity with $K_0$ this allows us to conclude that there is an integer $K$ for which $\nu_\ell(\gamma_i)>k$ whenever $\nu_2(\gamma_{i_0})\ge K$ for some $i\le i_0\le i+\ell-2$.

Notice that the right hand side of identity (\ref{contiden1}) can be evaluated just by knowing the value of the $(\ell-1)-$tuple $(\nu_2(\gamma_i),\ldots ,\nu_2(\gamma_{i+\ell-2}))$. Once these integers have been determined no further information is needed about $Q$ or $i$ in order to specify the value of $\nu_{\ell}(\gamma_i)$. Thus the integer $K$ which we found above can be chosen so that it depends only on $\ell$ and $k$. This allows us to conclude that for any $Q$ and $i$ if $\nu_\ell(\gamma_i)=k$ then the number of possibilities for the $(\ell-1)-$tuple $(\nu_2(\gamma_i),\ldots , \nu_2(\gamma_{i+\ell-2}))$ is at most $K^{\ell-1}$.
\subsection*{Step iii} Let us assume that $Q$ is large enough that $N(Q)>L(d)$. Then (\ref{Nksum1}) may be replaced by
\[N_{Q,d}(k)=\sum_{\ell=1}^{L}N_{Q,d}(k,\ell).\]
For each $1\le \ell\le L$ let $\{\vec{x}(\ell,m)\}_{1\le m\le n(\ell)}$ be the collection of possible $(\ell-1)-$tuples of integers from Step ii which can appear as the $2-$indices corresponding to an $\ell-$index which takes the value $k$. Then we can write $N_{Q,d}(k,\ell)$ as
\begin{align*}
\sum_{m=1}^{n(\ell)}\#\{1\le i\le N(Q) :& (q_{i-1},d)=(q_{i+\ell-1},d)=1,(q_{i+j-1},d)>1 \text{ for } 1\le j<\ell,\\
&(\nu_2(\gamma_i),\ldots ,\nu_2(\gamma_{i+\ell-2}))=\vec{x}(\ell,m)\}.
\end{align*}
If we write $\vec{x}(\ell,m)=(x_1(\ell,m),\ldots ,x_{\ell-1}(\ell,m))$ then if $(\nu_2(\gamma_i),\ldots ,\nu_2(\gamma_{i+\ell-2}))=\vec{x}(\ell,m)$ for some $i$ we have that
\begin{align*}
q_{i+j-1}=x_{j-1}(\ell,m)q_{i+j-2}-q_{i+j-3}
\end{align*}
for each $2\le j\le \ell$. Thus each $(\ell-1)-$tuple $\vec{x}(\ell,m)$ determines a pair of collections of residue classes $\A(\ell,m)$ and $\B(\ell,m)$ modulo $d$ with the property that
\[(q_{i-1},d)=(q_{i+\ell-1},d)=1,(q_{i+j-1},d)>1~\text{ for }~1\le j<\ell,\]
\[\text{ and }~(\nu_2(\gamma_i),\ldots ,\nu_2(\gamma_{i+\ell-2}))=\vec{x}(\ell,m)\]
if and only if
\[(q_{i-1},q_i)\mod d\in \A(\ell,m)\times\B(\ell,m),~\text{ and }~(\nu_2(\gamma_i),\ldots ,\nu_2(\gamma_{i+\ell-2}))=\vec{x}(\ell,m).\]
Furthermore it is obvious that $(a,d)=1$ for all $a\in\A(\ell ,m)$. Next using (\ref{Tkform1}) and (\ref{Tkform2}) we have that
\begin{equation*}
(\nu_2(\gamma_i),\ldots ,\nu_2(\gamma_{i+\ell-2}))=\vec{x}(\ell,m)
\end{equation*}
if and only if
\begin{equation*}
(q_{i-1},q_i)\in Q(\T_{x_1}\cap T^{-1}\T_{x_2}\cap\cdots \cap T^{-(\ell-2)}\T_{x_{\ell-1}})\cap\Zv.
\end{equation*}
Here we are implicitly using the well known fact that if $(a,b)\in Q\T\cap\Zv$ then $a=q_{i-1}$ and $b=q_i$ for some $i$. Putting all this together gives us the formula
\begin{align*}
N_{Q,d}(k)=\sum_{\ell=1}^{L}\sum_{m=1}^{n(\ell)}Z_{\A(\ell,m),\B(\ell,m),d}(Q(\T_{x_1}\cap \cdots \cap T^{-(\ell-2)}\T_{x_{\ell-1}})),
\end{align*}
where for ease of notation we have suppressed the dependence on $\ell$ and $m$. It is easy to check that each of the regions $Q(\T_{x_1}\cap T^{-1}\T_{x_2}\cap\cdots \cap T^{-(\ell-2)}\T_{x_{\ell-1}})$ is convex and so may apply Lemma \ref{latptlem1} to obtain the statement of the theorem with $c(d,k)$ equal to
\begin{align*}
\frac{6}{\pi^2d^2}\prod_{p|d}\left(1-\frac{1}{p^2}\right)^{-1}\sum_{\ell=1}^{L}\sum_{m=1}^{n(\ell)}|\A(\ell,m)||\B(\ell,m)|\ar (\T_{x_1}\cap \cdots \cap T^{-(\ell-2)}\T_{x_{\ell-1}}).
\end{align*}
Finally we note that some of the sets $\A(\ell,m)$ and $\B(\ell,m)$ and some of the regions $\T_{x_1}\cap T^{-1}\T_{x_2}\cap\cdots \cap T^{-(\ell-2)}\T_{x_{\ell-1}}$ may be empty, so to finish the proof we must show that $c(d,k)$ is never zero. To see this note that the set $\T_k$ is always non-empty. By Lemma \ref{latptlem1} if $Q$ is large enough there will be on the order of $Q^2$ lattice points $(a,b)\in Q\T_k\cap\Zv$ which satisfies $a\equiv 1\mod d$ and $b\equiv 0\mod d$. For each of these points we will have that $(kb-a,d)=1$ and they must therefore correspond to fractions $\gamma_i\in\F_Q$ with $(q_{i-1},d)=(q_{i+1},d)=1, (q_i,d)>1$, and $\nu_2(\gamma_i)=k$. Thus for $Q$ large enough $N_{Q,d}(k)$ is at least as large as a constant times $Q^2$. This precludes the possibility of having $c(d,k)$ equal to zero, and it completes our proof of Theorem \ref{mainthm}.

\vspace{10mm}

\noindent Dmitry A. Badziahin: Department of Mathematics, University of
York,

\vspace{0mm}

\noindent\phantom{Dmitry A. Badziahin: }Heslington, York, YO10 5DD,
England.


\noindent\phantom{Dmitry A. Badziahin: }e-mail: db528@york.ac.uk

\vspace{5mm}

\noindent Alan K. Haynes: Department of Mathematics, University of
York,

\vspace{0mm}

\noindent\phantom{Alan K. Haynes: }Heslington, York, YO10 5DD,
England.


\noindent\phantom{Alan K. Haynes: }e-mail: akh502@york.ac.uk

\vspace{5mm}

\end{document}